\documentclass[reqno,12pt]{amsart}%
\usepackage[left=2cm,right=2cm,top=2.5cm,bottom=2.5cm]{geometry}
\usepackage{xcolor}
\usepackage[colorlinks=true,linkcolor=blue]{hyperref}
\usepackage{eurosym}
\usepackage{amssymb}
\usepackage{amsmath}
\usepackage{amsfonts}
\usepackage{graphicx}
\usepackage{xypic}
\usepackage{epsfig}
\usepackage{stmaryrd}%
\providecommand{\U}[1]{\protect \rule{.1in}{.1in}}
\providecommand{\U}[1]{\protect \rule{.1in}{.1in}}
\theoremstyle{plain}
\newtheorem{theorem}{Theorem}[section]

\newtheorem{corollary}[theorem]{Corollary}

\newtheorem{definition}[theorem]{Definition}

\newtheorem{lemma}[theorem]{Lemma}

\newtheorem{proposition}[theorem]{Proposition}
\newtheorem{remark}[theorem]{Remark}

\numberwithin{equation}{section}
\begin{document}
\title[Lattice sequence spaces and summing mappings]{Lattice sequence spaces and summing mappings}
\author{Dahmane Achour }
\address{Laboratory of Functional Analysis and Geometry of Spaces, Faculty of
Mathematics and Computer Science. University Mohamed Boudiaf of M'sila-Algeria
PB 166 Ichebilia, M'sila 28000, Algeria\\
ID: https://orcid.org/0000-0003-0916-862X}
\email{dahmane.achour@univ-msila.dz}
\author{Toufik Tiaiba}
\address{Laboratory of Functional Analysis and Geometry of Spaces, Faculty of
Mathematics and Computer Science. University Mohamed Boudiaf of M'sila-Algeria
PB 166 Ichebilia, M'sila 28000, Algeria\\
ID: https://orcid.org/0009-0004-0040-5224}
\email{toufik.tiaiba@univ-msila.dz}
\thanks{The authors acknowledge with thanks the support of the General Direction of
Scientific Research and Technological Development (DGRSDT), Algeria.}

\begin{abstract}
The objective of this study is to advance the theory concerning positive
summing operators. Our focus lies in examining the space of positive strongly
$p$-summable sequences and the space of positive unconditionally $p$-summable
sequences. We utilize these in conjunction with the Banach lattice of positive
weakly $p$-summable sequences to present and characterize the classes of
positive strongly $(p,q)$-summing operators, positive $(p,q)$-summing, and
positive Cohen $(p,q)$-nuclear operators. Additionally, we describe these
classes in terms of the continuity of an associated tensor operator that is
defined between tensor products of sequences spaces.

\end{abstract}
\maketitle

\section{Introduction and background}

The spaces of sequences with values in a Banach lattice are intimately related
to the summability of operators between Banach lattices. For example, the
positive $(p,q)$-summing operators, introduced by Blasco\  \cite{OB}, are the
continuous operators which take positive weakly $p$-summable sequences
$\ell_{p,|\omega|}(X)$ into $q$-summable sequences $\ell_{q}(E)$ (see also
\cite{Zhu 98}). In \cite{achour p}, Achour-Belacel introduced the notion of
positive strongly $(p,q)$-summing operators to characterize those operators
whose adjoints are positive $(q^{\ast},p^{\ast})$-summing operators.

In \cite{Lab 04} and \cite{QG} the authors introduce the space of positive
strongly $p$-summable sequences $\ell_{p}^{\pi}(X)$ (initially introduced by
Cohen for Banach spaces \cite{Cohen}), as well as the space of positive
unconditionally $p$-summable sequences $\ell_{p,|\omega|}^{u}(X)$.

Tensor products have proved to be a useful tool for the theory of operator
ideals. Indeed, the excellent monograph \cite{tenso} deals with the tensor
product point of view of the theory and provides many applications to the
study of the structure of several spaces of summing linear operators.

The following characterizations provide nice examples of how tensor products
come into the theory of summing operators:\vspace{1mm}

\begin{itemize}
\item An operator $T:X\rightarrow Y$ is absolutely $p$-summing (see
\cite{Pie}) if and only if $I\otimes T:\ell_{p}\otimes_{\varepsilon
}X\rightarrow \ell_{p}\otimes_{\Delta_{p}}Y$ is continuous, where $\Delta_{p}$
satisfies $\Delta_{p}(\sum_{i=1}^{n}e_{i}\otimes x_{i})=(\sum_{i=1}^{n}\Vert
x_{i}\Vert^{p})^{1/p}$. and $\varepsilon$ is the injective tensor norm (see
\cite{tenso}).

\item Let $1<p<\infty$. An operator $T:X\rightarrow Y$ is Cohen $p$-nuclear
($p$-dominated) if and only if $I\otimes T:\ell_{p}\otimes_{\varepsilon
}X\rightarrow \ell_{p}\otimes_{\pi}Y$ is continuous, $\pi$ is the projective
norm (see \cite{Cohen} and \cite{tenso}).

\item An operator $T:X\rightarrow Y$ is strongly $p$-summing if and only if
$I\otimes T:\ell_{p}\otimes_{\Delta_{p}}X\rightarrow \ell_{p}\otimes_{\pi}Y$ is
continuous (see \cite{Alouani}).
\end{itemize}

The interplay between tensor products and positive summing operators have not
been explored yet. In this paper, first we utilize sequences in Banach lattice
spaces to define and characterize certain classes of positive summing
operators. Then we describe these classes in terms of the continuity of the
canonically defined tensor product operator $I\otimes T:\ell_{p}%
\otimes_{\alpha}X\rightarrow \ell_{p}\otimes_{\beta}Y$ for adequate $p$ and
tensor norms $\alpha$ and $\beta$.

Our results are presented as follows. After this introductory section, Section
2 is devoted to providing new properties of the positive $(p,q)$-summing
operators. Particularly, we prove that these operators are continuous
operators transform positive lattice unconditionally $p$-summable sequences
$\ell_{p,|\omega|}^{u}(X)$ into $q$-summable sequences $\ell_{q}(E)$. In
Section 3, utilizing the Banach lattice of positive strongly $p$-summable
sequences, we present a novel characterization of positive strongly
$(p,q)$-summing operators. Furthermore, we demonstrate that this class is
equivalent to the class of $(p,q)$-majorizing operators introduced in
\cite{chen2}. In Section 4, we study the notion of positive Cohen
$(p,q)$-nuclear operators. We explore the summability properties of these
operators by defining their corresponding operators between spaces of positive
weakly $p$-summable sequences $\ell_{p,|\omega|}(X)$ and strongly positive
strongly $p$-summable sequences $\ell_{p}^{\pi}(Y)$. In the final section, we
describe these classes in terms of the continuity of an associated tensor
operator that is defined between tensor products of sequences spaces.

We use standard notation for Banach lattices (see \cite{LT 96} and
\cite{Sch74}). If $X$ is an ordered set, the usual order on $X^{\mathbb{N}}$
is defined by $x=(x_{n})_{n}\geq0\Leftrightarrow x_{n}\geq0$ for each
$n\in \mathbb{N}$. If $X$ is a vector lattice and $\left \Vert x\right \Vert
\leq \left \Vert y\right \Vert $ whenever $\left \vert x\right \vert \leq \left \vert
y\right \vert $ ($\  \left \vert x\right \vert =\sup \left \{  x,-x\right \}  $) we
say that $X$ is a Banach lattice. If the lattice is complete, we say that $X$
is a Banach lattice. Note that this implies obviously that for any $x\in X$
the elements $x$ and $\left \vert x\right \vert $ have the same norm. We denote
by $X_{+}=\left \{  x\in X,x\geq0\right \}  .$ An element $x$ of $X$ is a
positive if $x\in X_{+}.$ For $x$ $\in X$ let $x^{+}:=\sup \left \{
x,0\right \}  ,x^{-}:=\sup \left \{  -x,0\right \}  $ be the positive part, the
negative part of $x$, respectively. For any $x\in X$, we have the following
properties $x=x^{+}-x^{-}$ and $\left \vert x\right \vert =x^{+}+x^{-}.$

The dual $X^{\ast}\ $of a Banach lattice $X$ is a complete Banach lattice
endowed with the natural order%

\[
x_{1}^{\ast}\leq x_{2}^{\ast}\Longleftrightarrow \left \langle x_{1}^{\ast
},x\right \rangle \leq \left \langle x_{2}^{\ast},x\right \rangle ,\  \forall x\in
X_{+}%
\]

\noindent where $\left \langle .,.\right \rangle $ denotes the bracket of duality.

By a sublattice of a Banach lattice $X$ we mean a linear subspace $A$ of
$X\ $so that $\sup \left \{  x,y\right \}  =x\wedge y$ belongs to $A$ whenever
$x,y\in A.$ The canonical embedding $i_{E}:X\rightarrow X^{\ast \ast}$ such
that $\left \langle i_{E}\left(  x\right)  ,x^{\ast}\right \rangle =\left \langle
x^{\ast},x\right \rangle $ of $X$ into its second dual $X^{\ast \ast}$ is an
order isometry from $X$ into a sublattice of $X^{\ast \ast}$, see
\cite[Proposition$\ $1.a.2]{LT 96}. If we consider $X$ as a sublattice of
$X^{\ast \ast}$ we have for $x_{1},x_{2}\in X$%

\[
x_{1}\leq x_{2}\Longleftrightarrow \left \langle x_{1},x^{\ast}\right \rangle
\leq \left \langle x_{2},x^{\ast}\right \rangle ,\  \forall x^{\ast}\in
X_{+}^{\ast}.
\]

\noindent Throughout this paper $X$ and $Y$ Banach lattices, $E$ and $F$
Banach spaces. An operator $T:X\rightarrow Y$ which preserves the lattice
operations is called lattice homomorphism, that is, $T(x_{1}\vee
x_{2})=T(x_{1})\vee T(x_{2})$ for all $x_{1},x_{2}\in X$. An one-to-one,
surjective lattice homomorphism is called lattice isomorphism. The space of
all bounded linear operators from $E$ to $F$ is denoted by $\mathcal{L}(E,F)$
and it is a Banach space with the usual supremum norm. The continuous dual
space $\mathcal{L}(E,\mathbb{K})$ of $E$ is denoted by $E^{\ast}$, whereas
$B_{E}$ denotes the closed unit ball of $E.$ The symbol $E\equiv F$ means that
$E$ and $F$ are isometrically isomorphic.

Let $1\leq p\leq \infty,$ we write $p^{\ast}$ the conjugate index of $p$, that
is $1/p+1/p^{\ast}=1$. As usual $\ell_{p}\left(  E\right)  $ denotes the
vector space of all absolutely $p$-summable sequences, with the usual norm
$\left \Vert \cdot \right \Vert _{p}$ and $\ell_{p,\omega}\left(  E\right)  $ the
space of all weakly $p$-summable sequences with the norm $\left \Vert \left(
x_{n}\right)  _{n}\right \Vert _{p,\omega}=\sup_{x^{\ast}\in B_{E^{\ast}}%
}\left \Vert (\left \langle x_{n},x^{\ast}\right \rangle )_{n}\right \Vert _{p}%
$.\ The closure in $\ell_{p,\omega}\left(  E\right)  $ of the set of all
sequences in $E$ which have only a finite number of non-zero terms, is a
Banach space with respect to the norm $\left \Vert \cdot \right \Vert _{p,\omega
}$. We denote this space by $\ell_{p}^{u}(E)$. The space

\begin{center}
$\left(  c_{0}\right)  _{\omega}(E):=\left \{  \left(  x_{n}\right)
_{n}\subset E:(\left \langle x_{n},\xi \right \rangle )_{n}\in c_{0},\forall
\xi \in E^{\ast}\right \}  $
\end{center}

\noindent is a closed subspace of $\ell_{\infty,\omega}\left(  E\right)  $
(see \cite[\S 19.4]{J}). It is well known that $\ell_{p,w}\left(  E\right)  $
is an isometrically isomorphic to $\mathcal{L}(\ell_{p^{\ast}},E)$ for
$1<p\leq \infty$ and $\ell_{1,\omega}\left(  E\right)  $ is an isometrically
isomorphic to $\mathcal{L}(c_{0},E).$ We denote by $\ell_{p}\left \langle
E\right \rangle \ $the space of all strongly $p$-summable sequences (Cohen
strongly $p$-summable sequences, see \cite{Cohen}), that is the space of all
sequences $\left(  x_{n}\right)  _{n=1}^{\infty}$ in $E$ such that
$(x_{n}^{\ast}(x_{n}))_{n}\in \ell_{1},$ for any $(x_{n}^{\ast})_{n=1}^{\infty
}\in \ell_{p^{\ast},\omega}(E^{\ast}),$ which is a Banach space with the norm
$\left \Vert \left(  x_{n}\right)  _{n=1}^{\infty}\right \Vert _{\left \langle
p\right \rangle }:=\underset{\left \Vert (x_{n}^{\ast})_{n=1}^{\infty
}\right \Vert _{p^{\ast},\omega}\leq1}{\sup}\left \vert \underset{n=1}%
{\overset{\infty}{\sum}}x_{n}^{\ast}(x_{n})\right \vert =\underset{\left \Vert
(x_{n}^{\ast})_{i=1}^{\infty}\right \Vert _{p^{\ast},\omega}\leq1}{\sup
}\underset{n=1}{\overset{\infty}{\sum}}\left \vert x_{n}^{\ast}(x_{n}%
)\right \vert .$ The following fact, discussed in \cite{Apiola}, is well-known%

\begin{equation}
\ell_{p^{\ast},\omega}(E^{\ast})\equiv \left[  \ell_{p}\left \langle
E\right \rangle \right]  ^{\ast}\text{ and }\left[  \ell_{p^{\ast}}\left \langle
E^{\ast}\right \rangle \right]  \equiv \left[  \ell_{p,\omega}(E)\right]
^{\ast}.\label{str-dual}%
\end{equation}
Moreover, it is well-known that%
\begin{equation}
\left[  \ell_{p}(E)\right]  ^{\ast}\equiv \ell_{p^{\ast}}(E^{\ast})\text{ for
}1\leq p<\infty \text{ and }\left[  c_{0}(E)\right]  ^{\ast}\equiv \ell
_{1}\left(  E^{\ast}\right)  .\label{b-dual}%
\end{equation}
\textit{Sequences in Banach lattice spaces. }Consider the case where $E$ is
replaced by a Banach lattice $X.$ The space of positive weakly $p$-summable
sequences was introduced in \cite{Lab 04} by%
\[
\ell_{p,|\omega|}(X)=\left \{  (x_{n})_{n}\in X^{\mathbb{N}}:\sum_{n=1}%
^{\infty}\langle x^{\ast},|x_{n}|\rangle^{p}<+\infty,~\forall x^{\ast}\in
X_{+}^{\ast}\right \}
\]
and
\[
\Vert(x_{n})_{n}\Vert_{p,|\omega|}=\sup_{x^{\ast}\in B_{X_{+}^{\ast}}}%
(\sum_{n=1}^{\infty}\langle x^{\ast},\left \vert x_{n}\right \vert \rangle
^{p})^{\frac{1}{p}}%
\]
Also, $\left(  c_{0}\right)  _{\left \vert \omega \right \vert }\left(  X\right)
$ is a closed vector lattice subspace of $\ell_{\infty,\left \vert
\omega \right \vert }\left(  X\right)  $. Then $(\ell_{p,|\omega|}(X),\Vert
\cdot \Vert_{p,|\omega|})$ is a Banach lattice (see \cite{QG} and \cite{Lab
04}). Moreover, we have\newline \noindent($I_{0}$) $\left \Vert \left(
x_{n}\right)  _{n}\right \Vert _{p,\omega}\leq \left \Vert \left(  x_{n}\right)
_{n}\right \Vert _{p,_{\left \vert \omega \right \vert }}$ for all $\left(
x_{n}\right)  \in \ell_{p,\left \vert \omega \right \vert }\left(  X\right)  $.

\noindent($I_{1}$) If $\left(  x_{n}\right)  _{n}\geq0,$ we have%

\begin{equation}
\left \Vert (x_{n})_{n}\right \Vert _{p,_{\left \vert \omega \right \vert }%
}=\left \Vert (x_{n})_{n}\right \Vert _{p,\omega}.\label{1.1}%
\end{equation}
We define%
\[
\ell_{p}^{\varepsilon}(X^{\ast})=\ell_{p,|\omega|}(X^{\ast})=\left \{
(x_{n}^{\ast})_{n}\in(X^{\ast})^{\mathbb{N}}:\sum_{n=1}^{\infty}\langle
x,|x_{n}^{\ast}|\rangle^{p}<+\infty,~~\forall x\in X_{+}\right \}
\]
and
\begin{equation}
\Vert(x_{n}^{\ast})_{n}\Vert_{p,|\omega|}=\sup_{x\in B_{X_{+}}}\left(
\sum_{n=1}^{\infty}\langle|x_{n}^{\ast}|,x\rangle^{p}\right)  ^{\frac{1}{p}%
}\label{1.2}%
\end{equation}
Then $\ell_{p,|\omega|}(X^{\ast})$ with this norm is a Banach lattice (see
\cite{QG} and \cite{Lab 04}).\newline Let $\ell_{p,|\omega|}^{u}(X)$ denote
the closed sublattice of $\ell_{p,|\omega|}(X)$ defined by
\[
\ell_{p,|\omega|}^{u}(X)=\left \{  (x_{n})_{n}\in X^{\mathbb{N}}:\lim_{n}%
\Vert(x_{k})_{k=n+1}^{\infty}\Vert_{p,|\omega|}=0\right \}  .
\]
In this case we say that $(x_{n})_{n}$ is positive unconditionally
$p$-summable sequences.\newline Let%
\[
\ell_{p}^{\pi}(X)=\left \{  (x_{n})_{n}\in X^{\mathbb{N}}:\sum_{n=1}^{\infty
}|\langle x_{n}^{\ast},|x_{n}|\rangle|<+\infty,~\forall(x_{n}^{\ast})_{n}%
\in(\ell_{p^{\ast},|\omega|}(X^{\ast}))^{+}\right \}
\]
and
\begin{equation}
\Vert(x_{n})_{n}\Vert_{\ell_{p}^{\pi}(X)}=\sup_{(x_{n}^{\ast})_{n}\in
B_{[\ell_{p^{\ast},|\omega|}(X^{\ast})]^{+}}}\sum_{n=1}^{\infty}\langle
x_{n}^{\ast},|x_{n}|\rangle.\label{ournorm}%
\end{equation}
In this case we say that $(x_{n})_{n}$ is positive strongly $p$-summable
sequences. Then $\ell_{p}^{\pi}(X)$ with this norm is a Banach lattice
(\cite{QG} ). For convenience let us denote $\ell_{1}\langle X\rangle=\ell
_{1}^{\pi}(X)=\ell_{1}(X).$\newline By (I$_{0}$) and (I$_{1}$), we have

\noindent($I_{0}^{\prime}$) $\left \Vert (x_{n})_{n}\right \Vert _{\ell_{p}%
^{\pi}(X)}\leq \left \Vert (x_{n})_{n}\right \Vert _{\ell_{p}\left \langle
X\right \rangle }$ for all $\left(  x_{n}\right)  _{n}\in \ell_{p}\left \langle
X\right \rangle $.

\noindent($I_{1}^{\prime}$) if $(x_{n})_{n}\geq0$ then $(x_{n})_{n}\in \ell
_{p}^{\pi}(X)$ if and only if $(x_{n})_{n}\in \ell_{p}\left \langle
X\right \rangle $, and%
\begin{equation}
\left \Vert (x_{n})_{n}\right \Vert _{\ell_{p}^{\pi}(X)}=\left \Vert (x_{n}%
)_{n}\right \Vert _{\ell_{p}\left \langle X\right \rangle }.
\end{equation}
Moreover, we have the following results due to \cite{bu, Lab 04}.

\begin{proposition}
\cite[Proposition 3.1 and Proposition 3.2]{bu} Let $X$ be a Banach lattice and
$1<p<\infty$

\begin{enumerate}
\item[(a)]
\[
\ell_{p}^{\pi}(X^{\ast})=\left \{  (x_{n}^{\ast})_{n}\in(X^{\ast})^{\mathbb{N}%
}:\sum_{n=1}^{\infty}|\langle x_{n},|x_{n}^{\ast}|\rangle|<+\infty
,\forall(x_{n})_{n}\in(\ell_{p^{\ast},|\omega|}(X))^{+}\right \}
\]
and for each $(x_{n}^{\ast})_{n}\in \ell_{p}^{\pi}(X^{\ast}),$
\[
\Vert(x_{n})_{n}\Vert_{\ell_{p}^{\pi}(X^{\ast})}=\sup_{(x_{n})_{n}\in
B_{[\ell_{p^{\ast},|\omega|}(X)]^{+}}}\sum_{n=1}^{\infty}\langle x_{n}%
,|x_{n}^{\ast}|\rangle.
\]

\item[(b)]
\[
\ell_{p}^{\pi}(X^{\ast})=\left \{  (x_{n}^{\ast})_{n}\in(X^{\ast})^{\mathbb{N}%
}:\sum_{n=1}^{\infty}|\langle x_{n},|x_{n}^{\ast}|\rangle|<+\infty
,\forall(x_{n})_{n}\in(\ell_{p^{\ast},|\omega|}^{u}(X))^{+}\right \}
\]
and for each $(x_{n}^{\ast})_{n}\in \ell_{p}^{\pi}(X^{\ast}),$
\[
\Vert(x_{n}^{\ast})_{n}\Vert_{\ell_{p}^{\pi}(X^{\ast})}=\sup_{(x_{n})_{n}\in
B_{[\ell_{p^{\ast},\omega}^{u}(X)]^{+}}}\sum_{n=1}^{\infty}\langle
x_{n},|x_{n}^{\ast}|\rangle.
\]

\end{enumerate}
\end{proposition}

\begin{lemma}
\label{lemma houm}\cite{Lab 04} Let $X$ be a Riesz space (i.e., vector
lattice), $(Y,C)$ an ordered vector space such that $Y=C-C$ and $T:X_{+}%
\longrightarrow C$ a positive, homogeneous and additive bijection. Then $Y$ is
a lattice space and $T$ can be uniquely extended to a lattice isomorphism from
$X$ onto $Y$.
\end{lemma}

\begin{theorem}
\label{lattice dual}Let $X$ be a Banach lattice and $1<p<\infty.$

\begin{enumerate}
\item[(i)] The Banach lattice $\ell_{p^{\ast},|\omega|}(X^{\ast})$ is lattice
and isometrically isomorphic to $\left[  \ell_{p}^{\pi}(X)\right]  ^{\ast}$

\item[(ii)] \cite[Corollary 3.3 and Corollary 3.4]{bu} The Banach lattice
$\left[  \ell_{p^{\ast}}^{\pi}(X^{\ast})\right]  $ is lattice and
isometrically isomorphic to $\left[  \ell_{p,|\omega|}^{u}(X)\right]  ^{\ast}$
\end{enumerate}
\end{theorem}

\begin{proof}
(i) Let $1<p<\infty$, we define the mapping%

\begin{align*}
T  &  :\ell_{p^{\ast},|\omega|}(X^{\ast})\longrightarrow \left[  \ell_{p}^{\pi
}(X)\right]  ^{\ast}\\
x^{\ast}  &  =(x_{n}^{\ast})_{n}\longmapsto T(x^{\ast})=T_{x^{\ast}}%
\end{align*}
where $T_{x^{\ast}}$ is the linear functional defined by
\begin{align*}
T_{x^{\ast}}  &  :\ell_{p}^{\pi}(X)\longrightarrow \mathbb{K}\\
(x_{n})_{n}  &  \longmapsto T_{x^{\ast}}((x_{n})_{n})=\sum_{n=1}^{\infty}%
x_{n}^{\ast}(x_{n}).
\end{align*}
The map $T$ is clearly a positive map from $\left(  \ell_{p^{\ast},|\omega
|}(X^{\ast})\right)  _{+}$ to $\left[  \ell_{p}^{\pi}(X)\right]  _{+}^{\ast}$,
and it is homogeneous, additive and injective.To see that it is surjective,
note that if $S\in \left[  \ell_{p}^{\pi}(X)\right]  _{+}^{\ast}$ and
\begin{align*}
I_{n}  &  :X\longrightarrow \ell_{p}\langle X\rangle \\
x  &  \longmapsto(0,...,x,0,...)
\end{align*}
$x_{n}^{\ast}=S\circ I_{n}\in X_{+}^{\ast}$ for all $n\in \mathbb{N}.$\newline
Then
\begin{align*}
T_{(S\circ I_{n})_{n}}((x_{n})_{n})  &  =\sum_{n=1}^{\infty}(S\circ
I_{n})(x_{n})\\
&  =\sum_{n=1}^{\infty}S(I_{n}(x_{n}))\\
&  =S(I_{1}(x_{1}))+\cdot \cdot \cdot+S(I_{n}(x_{n}))+\cdot \cdot \cdot \\
&  =S((x_{1},0,...))+\cdot \cdot \cdot+S((0,...,x_{n},...))+\cdot \cdot \cdot \\
&  =S((x_{n})_{n})
\end{align*}
For $x\in B_{X}^{+}$ , we get
\begin{align*}
\left(  \sum_{n=1}^{\infty}\langle x_{n}^{\ast},x\rangle|^{p^{\ast}}\right)
^{\frac{1}{p^{\ast}}}  &  =\left(  \sum_{n=1}^{\infty}|S\circ I_{n}%
)(x)|^{p^{\ast}}\right)  ^{\frac{1}{p^{\ast}}}\\
&  =\sup_{(\alpha_{n})\in B_{\ell_{p}}^{+}}|\sum_{n=1}^{\infty}\alpha
_{n}S\circ I_{n}(x)|\\
&  =\sup_{(\alpha_{n})\in B_{\ell_{p}}^{+}}|S((\alpha_{n}x)_{n})|\\
&  \leq \Vert S\Vert \sup_{(\alpha_{n})\in B_{\ell_{p}}^{+}}\Vert(\alpha
_{n}x)_{n}\Vert_{\ell_{p}\langle X\rangle}%
\end{align*}
We need to estimate the latter expression. Note that
\begin{align*}
\Vert(\alpha_{n}x)_{n}\Vert_{\ell_{p}\langle X\rangle}  &  =\sup_{(x_{n}%
^{\ast})_{n}\in B_{[\ell_{p^{\ast},|\omega|}(X^{\ast})]^{+}}}\sum_{n}\langle
x_{n}^{\ast},|\alpha_{n}x_{n}|\rangle \\
&  =\sup_{(x_{n}^{\ast})_{n}\in B_{[\ell_{p^{\ast},|\omega|}(X^{\ast})]^{+}}%
}\sum_{n}|\alpha_{n}|\langle x_{n}^{\ast},|x_{n}|\rangle \\
&  \leq \Vert(\alpha_{n})_{n}\Vert_{p}.\sup_{(x_{n}^{\ast})_{n}\in
B_{[\ell_{p^{\ast},|\omega|}(X^{\ast})]^{+}}}\Vert(x_{n})_{n}\Vert_{p^{\ast}%
}^{\omega}\\
&  \leq \Vert(\alpha_{n})_{n}\Vert_{p}%
\end{align*}
Then
\begin{align*}
\left(  \sum_{n=1}^{\infty}|S\circ I_{n})(x)|^{p^{\ast}}\right)  ^{\frac
{1}{p^{\ast}}}  &  \leq \Vert S\Vert \sup_{(\alpha_{n})\in B_{\ell_{p}}^{+}%
}\Vert(\alpha_{n})_{n}\Vert_{p}\\
&  =\Vert S\Vert
\end{align*}
Hence, by (\ref{1.2}), $(x_{n}^{\ast})_{n}\in(\ell_{p^{\ast},|\omega|}%
(X^{\ast}))_{+}.$

Since $\ell_{p^{\ast},|\omega|}(X^{\ast})$ is a Riesz space and $\left[
\ell_{p}^{\pi}(X)\right]  ^{\ast}=\left[  \ell_{p}^{\pi}(X)\right]  _{+}%
^{\ast}-\left[  \ell_{p}^{\pi}(X)\right]  _{+}^{\ast},$ it follows from Lemma
\ref{lemma houm} that $\left[  \ell_{p}^{\pi}(X)\right]  ^{\ast}$ is a lattice
space and that $T$ is a lattice isomorphism from $\ell_{p^{\ast},|\omega
|}(X^{\ast})$ onto $\left[  \ell_{p}^{\pi}(X)\right]  ^{\ast}.$\newline For
$(x_{n}^{\ast})_{n}\in \ell_{p^{\ast},|\omega|}(X^{\ast})$,
\begin{align*}
\Vert T((x_{n}^{\ast})_{n})\Vert_{\left[  \ell_{p}^{\pi}(X)\right]  ^{\ast}}
&  =\Vert T((x_{n}^{\ast})_{n})\Vert_{\mathcal{L}(\ell_{p}^{\pi}%
(X),\mathbb{K})}\\
&  =\Vert|T((x_{n}^{\ast})_{n})|\Vert_{\left[  \ell_{p}^{\pi}(X)\right]
^{\ast}}\\
&  =\Vert T(|(x_{n})_{n}|)\Vert_{\left[  \ell_{p}^{\pi}(X)\right]  ^{\ast}}\\
&  =\Vert(|x_{n}^{\ast}|)_{n}\Vert_{\ell_{p^{\ast},|\omega|}(X^{\ast})}\\
&  =\Vert(x_{n}^{\ast})_{n}\Vert_{\ell_{p^{\ast},|\omega|}(X^{\ast})}.
\end{align*}
This means that $T$ is an isometry from $\ell_{p^{\ast},|\omega|}(X^{\ast})$
onto $\left[  \ell_{p}^{\pi}(X)\right]  ^{\ast}.$
\end{proof}

\section{Positive $(p,q)$-summing operators generated by $\ell_{p,|\omega
|}^{u}(X)$}

Let $1\leq q\leq p<\infty.$ An operator $T:X\longrightarrow F$ is said to be
positive $(p,q)$-summing \cite{OB} if there exists a constant $C>0$ such that
for every $x_{1},x_{2},...,x_{n}\in X,$ we have
\begin{equation}
\left \Vert (T(x_{i}))_{i=1}^{n}\right \Vert _{\ell_{p}(F)}\leq C\left \Vert
(x_{i})_{i=1}^{n}\right \Vert _{\ell_{q,|\omega|}(X)}. \label{equ posi p-su}%
\end{equation}
For $q<p=\infty$,%
\[
\sup \Vert T(x_{i})\Vert \leq C\left \Vert (x_{i})_{i=1}^{n}\right \Vert
_{\ell_{q,|\omega|}(X)}.
\]
We shall denote by $\Lambda_{p,q}(X,F)$ the space of positive $(p,q)$-summing
operators. This space becomes a Banach space with the norm $\Vert \cdot
\Vert_{\Lambda_{p,q}}$ given by the infimum of the constants verifying
(\ref{equ posi p-su}). For $p=\infty$ and $1\leq q<\infty$ we consider
$\Lambda_{\infty,q}(X,Y)=\mathcal{L}(X,F)$ and $\Vert T\Vert_{\Lambda
_{\infty,q}}=\Vert T\Vert.$

Now, we give characterizations of these classes in terms of transformations of
lattice vector-valued sequences.

\begin{proposition}
\label{w-the}\cite[Proposition 2]{OB} Let $T:X\longrightarrow F$ be an
operator and $1\leq q\leq p\leq \infty$, the following are equivalent.

\begin{enumerate}
\item[(1)] $T\in \Lambda_{p,q}(X,F).$

\item[(2)] The associated operator $\widehat{T}:\ell_{q,|\omega|}%
(X)\longrightarrow \ell_{p}(F)$ given by $\widehat{T}\left(  (x_{i}%
)_{i=1}^{\infty}\right)  =(T(x_{i}))_{i=1}^{\infty},(x_{i})_{i=1}^{\infty}%
\in \ell_{q,|\omega|}(X)$ is well-defined and continuous.

In this case $\Vert T\Vert_{\Lambda_{p,q}}=\Vert \hat{T}\Vert.$
\end{enumerate}
\end{proposition}

\begin{theorem}
\label{u-the}From a continuous linear $T\in \mathcal{L}(X,F)$ and $1\leq q\leq
p\leq \infty$ , the following condition are equivalent.

(i) $T\in \Lambda_{p,q}(X,F).$

(ii) The sequence $(T(x_{i}))_{i}\in \ell_{p}(F)$ whenever $(x_{i})_{i}\in
\ell_{q,|\omega|}^{u}(X).$

(iii) The induced map%
\[
\widehat{T}:\ell_{q,|\omega|}^{u}(X)\rightarrow \ell_{p}\left(  F\right)
,\widehat{T}\left(  (x_{i})_{i=1}^{\infty}\right)  =(T(x_{i}))_{i=1}^{\infty}.
\]

is a well-defined continuous linear operator and $\Vert T\Vert_{\Lambda_{p,q}%
}=\Vert \widehat{T}\Vert$.
\end{theorem}

\begin{proof}
$(i)\Rightarrow(ii).$ Let $x=(x_{i})_{i=1}^{\infty}\in \ell_{q,|\omega|}%
^{u}(X)$, we have
\[
\left \Vert (T(x_{i}))_{i=1}^{n}\right \Vert _{\ell_{p}(F)}\leq C\left \Vert
(x_{i})_{i=1}^{n}\right \Vert _{\ell_{q,|\omega|}(X)},
\]
for all $n\in \mathbb{N}$. So, if $m_{1}>m_{2},$
\begin{align*}
\left \Vert (T(x_{i}))_{i=1}^{m_{1}}-(T(x_{i}))_{i=1}^{m_{2}}\right \Vert
_{\ell_{p}(F)} &  =\left \Vert (T(x_{i}))_{i=m_{2}+1}^{m_{1}}\right \Vert
_{\ell_{p}(F)}\\
&  \leq C\left \Vert (x_{i})_{i=m_{2}+1}^{m_{1}}\right \Vert _{\ell_{q,|\omega
|}(X)}.
\end{align*}
We conclude that $(y_{n})_{n=1}^{\infty}$ with $y_{n}=(T(x_{i}))_{i=1}^{n}$ is
Cauchy in $\ell_{p}(F)$ and so converges to some $(z_{i})_{i=1}^{\infty}%
\in \ell_{p}(F).$\newline Given $\varepsilon>0,$ we can find $N_{o}%
\in \mathbb{N}$ so that
\[
n\geq N_{0}\Rightarrow \left \Vert (T(x_{i}))_{i=1}^{n}-(z)_{i=1}^{\infty
}\right \Vert _{\ell_{p}(F)}<\varepsilon.
\]
So, for a fixed $i_{0}\in \mathbb{N},$ we have
\[
\Vert T(x_{i_{0}})-z_{i_{0}}\Vert<\varepsilon.
\]
We conclude that $T(x_{i_{0}})=z_{i_{0}}.$ Hence $(T(x_{i}))_{i=1}^{\infty
}=(z_{i})_{i=1}^{\infty}\in \ell_{p}(F).$

$(ii)\Rightarrow(iii)$. Is it clear that $T$ is linear implies that
$\widehat{T}$ is linear, for show that $\widehat{T}$ is continuous we show
that $\widehat{T}$ is a closed graph. Suppose that the sequence $(T(x_{i}%
)_{i=1}^{\infty})\in \ell_{p}(F)$ whenever $(x_{i})_{i=1}^{\infty}\in
\ell_{q,|\omega|}^{u}(X)$ and let $\left(  (x_{i},\widehat{T}(x_{i}))\right)
_{i=1}^{\infty}$ a convergent sequence in the Cartesian product $\ell
_{q,|\omega|}^{u}(X)\times \ell_{p}(F)$ that is, $\left(  (x_{i},\widehat
{T}(x_{i})\right)  \longrightarrow(x,y)$ . So
\begin{equation}
x_{k}\longrightarrow x=(x_{n})_{n=1}^{\infty}\in \ell_{q,|\omega|}%
^{u}(X),\label{25}%
\end{equation}
and
\begin{equation}
\hat{T}(x_{i})\longrightarrow y=(y_{n})_{n=1}^{\infty}\in \ell_{p}%
(F).\label{26}%
\end{equation}
From (\ref{25}), for all $\varepsilon>0$ there exist $k>0$ such that
\begin{align*}
|x^{\ast}(x_{i}^{k}-x_{i})|^{q} &  \leq(|x^{\ast}||(x_{i}^{k}-_{i})|)^{q}%
\leq \sum_{i=1}^{\infty}(x^{\ast}|(x_{i}^{k}-x_{i})|)^{q}\leq \sup_{x^{\ast}\in
B_{X_{+}^{\ast}}}\left(  \sum_{i=1}^{\infty}(x^{\ast}|(x_{i}^{k}-x_{i}%
)|)^{q}\right)  \\
&  \leq \Vert x^{k}-x\Vert_{\ell_{q,|\omega|}^{u}}^{q}\\
&  \leq \varepsilon^{q}%
\end{align*}
whenever $k>k_{0}$, $x^{\ast}\in B_{X_{+}^{\ast}}$ and for all $i\in
\mathbb{N}$ in this way, by Hanh-Banach Theorem, we get
\begin{equation}
\Vert x_{i}^{k}-x_{i}\Vert^{q}=\sup_{x^{\ast}\in B_{X^{\ast}}}|x^{\ast}%
(x_{i}^{k}-x_{i})|^{q}\leq2^{q}\sup_{x^{\ast}\in B_{X_{+}^{\ast}}}|x^{\ast
}(x_{i}^{k}-x_{i})|^{q}\leq2^{q}\xi^{q}.\label{27}%
\end{equation}
whenever $k>k_{0}$ and for all $i\in \mathbb{N}$, then we have $x_{i}%
^{k}\longrightarrow x_{i}\in X$ for all $k\longrightarrow \infty.$ How $T$ is
continuous, we find
\[
\lim_{k}T(x_{i}^{k})=T(x_{i}),\  \text{for all }\  \ i\in \mathbb{N}%
\]
From (\ref{26}), for all $\varepsilon>0$ there exist $k^{^{\prime}}>0$ such
that
\begin{align*}
\Vert T(x_{i}^{k^{^{\prime}}})-y_{i}\Vert^{p} &  \leq \sum_{i=1}^{\infty}\Vert
T(x_{i}^{k^{^{\prime}}})-y_{i}\Vert^{p}\leq \Vert(T(x_{i}^{k^{^{\prime}}%
}))_{i=1}^{\infty}-(y_{i})_{i=1}^{\infty}\Vert_{p}^{p}\\
&  =\Vert \widehat{T}(x^{k^{^{\prime}}})-y\Vert_{p}^{p}\\
&  \leq \varepsilon^{p},
\end{align*}
whenever $k^{^{\prime}}>k_{0}^{^{\prime}}$ and for all $i\in \mathbb{N}$, we
find
\begin{equation}
\lim_{k}T(x_{i}^{k^{^{\prime}}})=y_{i},\  \  \text{for all }\  \ i\in
\mathbb{N}.\label{28}%
\end{equation}
From (\ref{27}), (\ref{28}) and uniqueness of the limit, it follows that
\[
\widehat{T}(x)=(T(x_{i})_{i=1}^{\infty})=(y_{i})_{i=1}^{\infty}=y.
\]
This implies that the linear mapping $\widehat{T}$ has a closed graph.

$(iii)\Rightarrow(i)$ Is straightforward.

\end{proof}

\section{Positive strongly $(p,q)$-summing operators generated by $\ell
_{p}^{\pi}(Y)$}

Let $1\leq q\leq p\leq \infty$ \ Recall that an operator $T\in \mathcal{L}(E,Y)$
is called positive strongly $(p,q)$-summing \cite{achour p} if there exists a
constant $C>0$ such that for all finite sets $n\in \mathbb{N}$, $(x_{i}%
)_{i=1}^{n}\subset E$ and $(y_{i}^{\ast})_{i=1}^{n}\subset(Y^{\ast})^{+}$, we
have
\begin{equation}
\sum_{i=1}^{n}|\langle T(x_{i}),y_{i}^{\ast}\rangle|\leq C\Vert(x_{i}%
)_{i=1}^{n}\Vert_{q}.\Vert(y_{i}^{\ast})_{i=1}^{n}\Vert_{p^{\ast},\omega}.
\label{st-def}%
\end{equation}
We shall denote by $\mathcal{D}_{p,q}^{+}(E,Y)$ the space of positive strongly
$(p,q)$-summing operators or $\mathcal{D}_{p}^{+}(E,Y)$ if $p=q$, the space of
positive strongly $p$-summing operators. This space becomes a Banach space
with the norm $\Vert \cdot \Vert_{\mathcal{D}_{p,q}^{+}}$ given by the infimum
of the constants verifying (\ref{st-def}).\newline

\begin{lemma}
$T\in \mathcal{D}_{p,q}^{+}(E,Y)$ if and only if there exists a constant $C>0$
such that for all finite sets $n\in \mathbb{N}$, $(x_{i})_{i=1}^{n}\subset E$
and $(y_{i}^{\ast})_{i=1}^{n}\subset(Y^{\ast})^{+},$ we have
\begin{equation}
\Vert(T(x_{i})_{i=1}^{n}\Vert_{\ell_{p}^{\pi}(Y)}\leq C\Vert(x_{i})_{i=1}%
^{n}\Vert_{q}. \label{30}%
\end{equation}

\end{lemma}

\begin{proof}
Let $T\in \mathcal{D}_{p,q}^{+}(E,Y)$ then there exists a constant $C>0$ such
that for all finite sets $n\in \mathbb{N}$, $(x_{i})_{i=1}^{n}\subset E$ and
$(y_{i}^{\ast})_{i=1}^{n}\subset(Y^{\ast})^{+},$ we have%

\[
\sum_{i=1}^{n}|\langle T(x_{i}),y_{i}^{\ast}|\rangle \leq C\Vert(x_{i}%
)_{i=1}^{n}\Vert_{q}.\Vert(y_{i}^{\ast})_{i=1}^{n}\Vert_{p^{\ast},\omega}.
\]
For each $z\in Y$ and $z^{\ast}\in Y^{\ast},$
\begin{equation}
|z^{\ast}|(|z|)=\sup \left \{  |g^{\ast}(z)|:|g^{\ast}|\leq|z^{\ast}|\right \}  .
\label{31}%
\end{equation}

Now let $(y_{i}^{\ast})_{i=1}^{n}\in(\ell_{p^{\ast},|\omega|}(Y^{\ast}))^{+}$
and $\varepsilon>0$, from (\ref{31}) there exists for each $1\leq i\leq n$, a
$g_{i}^{\ast}\in Y^{\ast}$ such that $|g_{i}^{\ast}|\leq|y_{i}^{\ast}|$ and
\[
|y_{i}^{\ast}|(|T(x_{i})|)\leq|g_{i}^{\ast}(T(x_{i}))|+\frac{\varepsilon}{n}.
\]
Note that $(g_{i}^{\ast})_{i=1}^{n}\in(\ell_{p^{\ast},|\omega|}^{n}(Y^{\ast
}))^{+}.$ Then
\begin{align*}
\sum_{i=1}^{n}\langle|T(x_{i})|,y_{i}^{\ast}\rangle &  =\sum_{i=1}^{n}%
\langle|T(x_{i})|,|y_{i}^{\ast}|\rangle \leq \sum_{i=1}^{n}|\langle
T(x_{i}),g_{i}^{\ast}\rangle|+\varepsilon \\
&  \leq C\left[  \Vert(x_{i})_{i=1}^{n}\Vert_{q}.\Vert(g_{i}^{\ast})_{i=1}%
^{n}\Vert_{p^{\ast},\omega}\right]  +\varepsilon \\
&  \leq C\left[  \Vert(x_{i})_{i=1}^{n}\Vert_{q}.\Vert(y_{i}^{\ast})_{i=1}%
^{n}\Vert_{p^{\ast},\omega}\right]  +\varepsilon.
\end{align*}
Then
\[
\Vert(T(x_{i})_{i=1}^{n}\Vert_{\ell_{p}^{\pi}(Y)}\leq C\Vert(x_{i})_{i=1}%
^{n}\Vert_{q}+\varepsilon.
\]
Conversely, directly by $|\langle T(x_{i}),y_{i}^{\ast}\rangle|\leq
\langle|T(x_{i})|,y_{i}^{\ast}\rangle$ for every $i.$
\end{proof}

As in classical cases, the natural approach to presenting the summability
properties of positive strongly $(p,q)$-summing operators by defining the
corresponding operator between appropriate lattice sequence spaces..

\begin{proposition}
\label{c-the}Let $T:E\longrightarrow Y$ be an operator and $1\leq q\leq
p\leq \infty$, the following are equivalent:

\begin{enumerate}
\item[(1)] $T\in \mathcal{D}_{p,q}^{+}(E,Y).$

\item[(2)] The associated operator $\widehat{T}:\ell_{q}(E)\longrightarrow
\ell_{p}^{\pi}(Y)$ given by $\widehat{T}\left(  (x_{i})_{i=1}^{\infty}\right)
=(T(x_{i}))_{i=1}^{\infty},(x_{i})_{i=1}^{\infty}\in \ell_{q}(E)$ is
well-defined and continuous.

In this case $\Vert T\Vert_{\mathcal{D}_{p,q}^{+}}=\Vert \widehat{T}\Vert.$
\end{enumerate}
\end{proposition}

\begin{proof}
For the necessity, let $T\in \mathcal{D}_{p,q}^{+}(E,Y),$ $\left(
x_{i}\right)  _{i=1}^{\infty}\in \ell_{q}(E)$ and $\left(  y_{i}^{\ast}\right)
_{i=1}^{\infty}\in(\ell_{p^{\ast},|\omega|}(Y^{\ast}))^{+}.$ So we have%
\[%
\begin{array}
[c]{lll}%
\sum_{i=1}^{\infty}\left \langle |T(x_{i})|,y_{i}^{\ast}\right \rangle  & = &
\underset{n}{\sup}\sum_{i=1}^{n}\left \langle |T(x_{i})|,y_{i}^{\ast
}\right \rangle \\
& \leq & \Vert T\Vert_{\mathcal{D}_{p,q}^{+}}\underset{n}{\sup}\Vert
(x_{i})_{i=1}^{n}\Vert_{q}.\Vert(y_{i}^{\ast})_{i=1}^{n}\Vert_{p^{\ast}%
,\omega}\\
& \leq & \Vert T\Vert_{\mathcal{D}_{p,q}^{+}}\Vert(x_{i})_{i=1}^{\infty}%
\Vert_{q}.\Vert(y_{i}^{\ast})_{i=1}^{\infty}\Vert_{p^{\ast},\omega}.
\end{array}
\]
Since this holds for all $\left(  y_{i}^{\ast}\right)  _{i=1}^{\infty}\in
(\ell_{p^{\ast},|\omega|}(Y^{\ast}))^{+}$, we obtain
\[
\Vert(T(x_{i}))_{i=1}^{\infty}\Vert_{\ell_{p}^{\pi}(Y)}=\underset{\left \Vert
\left(  y_{i}^{\ast}\right)  _{i=1}^{\infty}\right \Vert _{p^{\ast},\omega}%
\leq1}{\sup}\sum_{i=1}^{\infty}\left \langle |T(x_{i})|,y_{i}^{\ast
}\right \rangle \leq \Vert T\Vert_{\mathcal{D}_{p,q}^{+}}\left \Vert \left(
x_{i}\right)  _{i=1}^{\infty}\right \Vert _{q},
\]
and then $\widehat{T}$ is continuous with norm $\leq \Vert T\Vert
_{\mathcal{D}_{p,q}^{+}}$.

In order to prove sufficiency, suppose $\widehat{T}$ is well-defined and
continuous and assume that $T\notin \mathcal{D}_{p,q}^{+}(E,Y).$ Then for each
$n\in \mathbb{N},$ we may choose a finite sequence $\left(  x_{i,n}\right)
_{i=1}^{m_{n}}\subset E$ such that
\[%
\begin{array}
[c]{ccc}%
\left \Vert \left(  x_{i,n}\right)  _{i=1}^{m_{n}}\right \Vert _{q}\leq1 &
\text{and} & \Vert(T(x_{i,n}))_{i=1}^{m_{n}}\Vert_{\ell_{p}^{\pi}(Y)}>2^{n},
\end{array}
\]

which implies
\begin{equation}
\sum_{i=1}^{m_{n}}\langle|T(x_{i,n})|,y_{i,n}^{\ast}\rangle>2^{2n}
\label{etoile5}%
\end{equation}
for some $(y_{i,n}^{\ast})_{i=1}^{m_{n}}\in(\ell_{p^{\ast},|\omega|}(Y^{\ast
}))^{+}$ such that $\left \Vert (y_{i,n}^{\ast})_{i=1}^{m_{n}}\right \Vert
_{p^{\ast},\omega}\leq1.$ Let $(z_{j})_{j=1}^{\infty}$ be the sequence
\begin{align*}
&  \left(  \left(  \frac{x_{i,1}}{2^{1}}\right)  _{i=1}^{m_{1}},\left(
\frac{x_{i,2}}{2^{2}}\right)  _{i=1}^{m_{2}},...,\left(  \frac{x_{i,n}}{2^{n}%
}\right)  _{i=1}^{m_{n}},...\right) \\
&  =\left(  \frac{x_{1,1}}{2^{1}},\frac{x_{2,1}}{2^{1}},...,\frac{x_{m_{1},1}%
}{2^{1}},\frac{x_{1,2}}{2^{2}},\frac{x_{2,2}}{2^{2}},...,\frac{x_{m_{2},2}%
}{2^{2}},...,\frac{x_{1,n}}{2^{n}},\frac{x_{2,n}}{2^{n}},...,\frac{x_{m_{n}%
,n}}{2^{n}},...\right)  .
\end{align*}
We have
\[
\left \Vert (z_{j})_{j}^{\infty}\right \Vert _{q}=\left(  \sum_{j=1}^{+\infty
}\sum_{i=1}^{m_{j}}\left \Vert \frac{x_{i,j}}{2^{j}}\right \Vert ^{q}\right)
^{\frac{1}{q}}=\left(  \sum_{j=1}^{+\infty}\frac{1}{2^{jq}}\left \Vert \left(
x_{i,j}\right)  _{i=1}^{m_{j}}\right \Vert _{q}^{q}\right)  ^{\frac{1}{q}}%
\leq \left(  \sum_{j=1}^{+\infty}\frac{1}{2^{jq}}\right)  ^{\frac{1}{q}}\leq1.
\]
Then, $(z_{j})_{j=1}^{\infty}\in \ell_{q}(E)$. However, $\widehat{T}%
((z_{j})_{j=1}^{\infty})\notin \ell_{p}^{\pi}(Y).$ In order to see this,
consider the sequences
\[
(\varphi_{j})_{j=1}^{\infty}=\left(  \left(  \frac{y_{i,1}^{\ast}}{2^{1}%
}\right)  _{i=1}^{m_{1}},\left(  \frac{y_{i,2}^{\ast}}{2^{2}}\right)
_{i=1}^{m_{2}},...,\left(  \frac{y_{i,n}^{\ast}}{2^{n}}\right)  _{i=1}^{m_{n}%
},...\right)  .
\]
Clearly $(\varphi_{j})_{j=1}^{\infty}\in B_{(\ell_{p^{\ast},|\omega|}(Y^{\ast
}))^{+}}.$ Then
\begin{align*}
\left \Vert (\varphi_{j})_{j=1}^{\infty}\right \Vert _{p^{\ast},|\omega|}  &
=\sup_{y\in B_{Y_{+}}}\left(  \sum_{j=1}^{+\infty}\sum_{i=1}^{m_{j}}\langle
x,\frac{y_{i,j}^{\ast}}{2^{j}}\rangle^{p^{\ast}}\right)  ^{\frac{1}{p^{\ast}}%
}=\sup_{y\in B_{Y_{+}}}\left(  \sum_{j=1}^{+\infty}\frac{1}{2^{jp^{\ast}}}%
\sum_{i=1}^{m_{j}}\langle x,y_{i,j}^{\ast}\rangle^{p^{\ast}}\right)
^{\frac{1}{p^{\ast}}}\\
&  \leq \left(  \sum_{j=1}^{+\infty}\frac{1}{2^{jp^{\ast}}}\right)  ^{\frac
{1}{p^{\ast}}}\leq1.
\end{align*}
By (\ref{etoile5}) it turns out that%
\begin{align*}
\Vert \widehat{T}((z_{j})_{j=1}^{\infty})\Vert_{\ell_{p}^{\pi}(Y)}  &
=\Vert(T(z_{j}))_{j=1}^{\infty})\Vert_{\ell_{p}^{\pi}(Y)}=\underset{\left \Vert
\left(  \xi_{j}\right)  _{j=1}^{\infty}\right \Vert _{p^{\ast},\omega}\leq
1}{\sup}\sum_{j=1}^{\infty}\langle|T(z_{j})|,\xi_{j}\rangle \\
&  \geq \sum_{j=1}^{\infty}\langle|T(z_{j})|,\varphi_{j}\rangle \\
&  =\sum_{j=1}^{\infty}\frac{1}{2^{2j}}\sum_{i=1}^{m_{j}}\langle
|T(x_{i,j})|,y_{i,j}^{\ast}\rangle=\infty
\end{align*}
which according (\ref{ournorm}) is a contradiction with the fact that
$\widehat{T}$ maps $\ell_{q}(X)$ (continuously) into $\ell_{p}^{\pi}(Y).$
Since%
\[
\Vert(T(x_{i}))_{i=1}^{\infty})\Vert_{\ell_{p}^{\pi}(Y)}=\Vert \widehat
{T}((x_{i})_{i=1}^{\infty})\Vert_{\ell_{p}^{\pi}(Y)}\leq \Vert \widehat{T}%
\Vert \left \Vert \left(  x_{i}\right)  _{i=1}^{\infty}\right \Vert _{q},
\]

we have $\Vert T\Vert_{\mathcal{D}_{p,q}^{+}}\leq \Vert \widehat{T}\Vert.$
\end{proof}

In the following result, we characterize the class of positive summing linear
operators and positive strongly summing linear operators by utilizing the
adjoint operator. For the proof of this result, we will utilize the duality of
lattice sequence spaces. Theorem \ref{adjo theorem} was established in \cite[
Theorem 4.6]{achour p}, and the proof provided there is direct. Using Theorem
\ref{lattice dual}, (\ref{b-dual}), Proposition \ref{c-the} and taking into
account that the adjoint of the $\widehat{T}:\ell_{q}(E)\longrightarrow
\ell_{p}^{\pi}(Y)$ can be identified with the operator $\widehat{T^{\ast}%
}:\ell_{p^{\ast},|\omega|}(Y^{\ast})\longrightarrow \ell_{q^{\ast}}(E^{\ast
});\widehat{T^{\ast}}((y_{i}^{\ast})_{i})=\left(  T^{\ast}(y_{i}^{\ast
})\right)  _{i},$ we provide an alternative proof of the results in Theorem
\ref{adjo theorem}.

\begin{theorem}
\label{adjo theorem}Let $T:E\longrightarrow Y$ be an operator and $1\leq q\leq
p\leq \infty.$

\begin{enumerate}
\item[(1)] The operator $T$ belongs to $\Lambda_{p,q}(X,F)$ if and only if its
adjoint $T^{\ast}$ belongs to $\mathcal{D}_{q^{\ast},p^{\ast}}^{+}(F^{\ast
},X^{\ast})$. Furthermore, $\Vert T\Vert_{\Lambda_{p,q}}=\Vert T^{\ast}%
\Vert_{\mathcal{D}_{q^{\ast},p^{\ast}}^{+}}.$

\item[(2)] The operator $T$ belongs to $\mathcal{D}_{p,q}^{+}(E,Y)$ if and
only if its adjoint $T^{\ast}$ belongs to $\Lambda_{q^{\ast},p^{\ast}}%
(Y^{\ast},E^{\ast})$. Furthermore, $\Vert T\Vert_{\mathcal{D}_{p,q}^{+}}=\Vert
T^{\ast}\Vert_{\Lambda_{q^{\ast},p^{\ast}}}.$
\end{enumerate}
\end{theorem}

\begin{proof}
(1) Let $T\in \mathcal{L}(X,F)$ and $T^{\ast}\in \mathcal{L}(F^{\ast},X^{\ast})$
its adjoint. Suppose that $T\in \Lambda_{p,q}(X,F),$ then by Theorem
(\ref{u-the}) $\widehat{T}:\ell_{q,|\omega|}^{u}(X)\longrightarrow \ell_{p}(F)$
\ is continuous with $\Vert T\Vert_{\Lambda_{p,q}}=\Vert \widehat{T}\Vert.$ By
(\ref{b-dual}) and Theorem (\ref{lattice dual}), the following diagram
commutes
\[%
\begin{array}
[c]{ccc}%
\ell_{p}\left(  F\right)  ^{\ast} & \overset{\widehat{T}^{\ast}}%
{\longrightarrow} & \ell_{q,|\omega|}^{u}(X)^{\ast}\\
J_{1}\uparrow &  & \uparrow J_{2}\\
\ell_{p^{\ast}}(F^{\ast}) & \overset{\widehat{T^{\ast}}}{\longrightarrow} &
\ell_{q^{\ast}}^{\pi}(X^{\ast})
\end{array}
\]
i.e.~$\widehat{T}^{\ast}\circ J_{1}=J_{2}\circ \widehat{T^{\ast}},$ where
$J_{1}$ is an isometric isomorphism and $J_{2}$ is an isometric lattice
isomorphism such that $J_{i}((z_{n}^{\ast})_{n})((z_{n})_{n})=f((z_{n})_{n})=%
{\displaystyle \sum \limits_{n=1}^{\infty}}
\left \langle z_{n},z_{n}^{\ast}\right \rangle ,i=1,2$, with the inverse $I_{i}%
$, defined by $I_{i}\left(  f\right)  =\left(  f\circ I_{n}\right)
_{n}=\left(  z_{n}^{\ast}\right)  _{n}$. In fact, the map $\widehat{T^{\ast}}%
$, defined by $\left(  y_{n}^{\ast}\right)  _{n}\mapsto \left(  T^{\ast}\left(
y_{n}^{\ast}\right)  \right)  _{n}$, let $\left(  y_{n}^{\ast}\right)  _{n}%
\in \ell_{p^{\ast}}(F^{\ast})$, then for all $(x_{n})_{n}\in \ell_{q,|\omega
|}^{u}(X),$%
\begin{align*}
(\widehat{T}^{\ast}\circ J_{1})((y_{n}^{\ast})_{n})\left(  (x_{n})_{n}\right)
&  =J_{1}((y_{n}^{\ast})_{n})\left(  \widehat{T}\left(  (x_{n})_{n}\right)
\right)  \\
&  =J_{1}((y_{n}^{\ast})_{n})\left(  (T(x_{n}))_{n}\right)  \\
&  =%
{\displaystyle \sum \limits_{n=1}^{\infty}}
\left \langle y_{n}^{\ast},T(x_{n})\right \rangle \\
&  =%
{\displaystyle \sum \limits_{n=1}^{\infty}}
\left \langle T^{\ast}(y_{n}^{\ast}),x_{n}\right \rangle \\
&  =J_{2}((T^{\ast}(y_{n}^{\ast}))_{n})\left(  (x_{n})_{n}\right)  \\
&  =J_{2}(\widehat{T^{\ast}}\left(  \left(  y_{n}^{\ast}\right)  _{n}\right)
\left(  (x_{n})_{n}\right)  \\
&  =(J_{2}\circ \widehat{T^{\ast}})\left(  \left(  y_{n}^{\ast}\right)
_{n}\right)  \left(  (x_{n})_{n}\right)  .
\end{align*}
i.e., $\widehat{T}^{\ast}\circ J_{1}=J_{2}\circ \widehat{T^{\ast}}.$ Then,
$\widehat{T}$ is well-defined and continuous if and only if $\widehat{T^{\ast
}}$ is well-defined and continuous. Consequently, from Proposition
\ref{c-the}, it follows that $T$ is positive $\left(  p,q\right)  $-summing,
if and only if its adjoint $T^{\ast}\in \mathcal{L}(F^{\ast},X^{\ast})$ is
strongly positive $\left(  q^{\ast},p^{\ast}\right)  $-summing. Furthermore,
$\Vert T\Vert_{\Lambda_{p,q}}=\Vert T^{\ast}\Vert_{\mathcal{D}_{q^{\ast
},p^{\ast}}^{+}}=\Vert \widehat{T}\Vert.$

(2) Let $T\in \mathcal{D}_{p,q}^{+}(E,Y).$ Then, by Proposition (\ref{c-the}),
the operator $\widehat{T}:\ell_{q}(E)\longrightarrow \ell_{p}^{\pi}(Y)$\ is
continuous with $\Vert T\Vert_{\mathcal{D}_{p,q}^{+}}=\Vert \widehat{T}\Vert.$
Using Theorem (\ref{lattice dual}), and taking into account that the adjoint
of the operator $\widehat{T}:\ell_{q}(E)\longrightarrow \ell_{p}^{\pi}(Y)$ can
be identified with the operator%
\[
\widehat{T^{\ast}}:\ell_{p^{\ast},|\omega|}(Y^{\ast})\longrightarrow
\ell_{q^{\ast}}(E^{\ast})\text{ given by }\widehat{T^{\ast}}((y_{i}^{\ast
})_{i})=\left(  T^{\ast}(y_{i}^{\ast})\right)  _{i},
\]
it follows that $\widehat{T}$ and $\widehat{T^{\ast}}$ are well-defined and
continuous. Therefore, from Proposition (\ref{w-the}) it follows that $T$ is
positive strongly $(p,q)$-summing if and only if its adjoint $T^{\ast}$ is
positive $(q^{\ast},p^{\ast})$-summing, and $\Vert T\Vert_{\mathcal{D}%
_{p,q}^{+}}=\Vert T^{\ast}\Vert_{\Lambda_{q^{\ast},p^{\ast}}}=\Vert \widehat
{T}\Vert.$
\end{proof}

\begin{corollary}
Let $1\leq q\leq p\leq \infty$.\newline

\begin{enumerate}
\item[(1)] The operator $T\in \mathcal{L}(E,Y)$ belongs to $\mathcal{D}%
_{p,q}^{+}(E,Y)$ if and only if $T^{\ast \ast}$ belongs to $\mathcal{D}%
_{p,q}^{+}(E^{\ast \ast},Y^{\ast \ast})$. Furthermore,
\[
\Vert T\Vert_{\mathcal{D}_{p,q}^{+}}=\Vert T^{\ast \ast}\Vert_{\mathcal{D}%
_{p,q}^{+}}.
\]

\item[(2)] The operator $T\in \mathcal{L}(X,F)$ belongs to $\Lambda_{p,q}(X,F)$
if and only if $T^{\ast \ast}$ belongs to $\Lambda_{p,q}(X^{\ast \ast}%
,F^{\ast \ast})$ . Furthermore,
\[
\Vert T\Vert_{\Lambda_{p,q}}=\Vert T^{\ast \ast}\Vert_{\Lambda_{p,q}}.
\]

\end{enumerate}
\end{corollary}

We say that an operators $T:E\longrightarrow Y$ is called positive
$(p,q)$-majorizing (see \cite{chen2} for $p=q$) if there exists a constant
$C>0$ such that%

\begin{equation}
\left(  \sum_{i=1}^{n}|\langle T(z_{i}),y_{i}^{\ast}\rangle|^{q^{\ast}%
}\right)  ^{\tfrac{1}{q^{\ast}}}\leq C\Vert(y_{i}^{\ast})_{i=1}^{n}%
\Vert_{p^{\ast},\omega}\label{aqu majori}%
\end{equation}
for all $(z_{i})_{i=1}^{n}$ in $B_{E}$ and $(y_{i}^{\ast})_{i=1}^{n}$ in
$(Y^{\ast})^{+}$. The space of all positive $(p,q)$-majorizing from $E$ to $Y$
is denoted by $\Upsilon_{p,q}(E,Y)$. this space becomes a Banach space with
the norm $\Vert \cdot \Vert_{\Upsilon_{p,q}}$ given by the infimum of the
constants $C$ satisfying (\ref{aqu majori}). In \cite{chen2}, the authors
proved the duality relationships between positive $p$-summing operators and
positive $p$-majorizing operators. It was known \cite{achour p} that an
operator $T:X\longrightarrow F$ is positive $p$-summing if and only if
$T^{\ast}$ is positive strongly $p^{\ast}$-summing. Similarly, an operator
$T:E\longrightarrow Y$ is positive strongly $p$-summing if and only if
$T^{\ast}$ is positive $p^{\ast}$-summing. In the following, we directly prove
that the concept of positive strongly $p$-summing and the concept of positive
$p$-majorizing are equivalent.

\begin{theorem}
Let $T:E\longrightarrow Y$ be an operator, the following are equivalent:

\begin{enumerate}
\item[(1)] $T$ is positive $(p,q)$-majorizing.

\item[(2)] $T$ is positive strongly $(p,q)$-summing.
\end{enumerate}
\end{theorem}

\begin{proof}
Suppose that $T$ is positive $(p,q)$-majorizing, given any finite sequence
$(x_{i})_{i=1}^{n}$ in $E$ and $(y_{i}^{\ast})_{i=1}^{n}$ in $(Y^{\ast})^{+}$,
we get
\begin{align*}
\sum_{i=1}^{n}|\langle T(x_{i}),y_{i}^{\ast}\rangle| &  =\sum_{i=1}^{n}\Vert
x_{i}\Vert|\langle T(\frac{x_{i}}{\Vert x_{i}\Vert}),y_{i}^{\ast}\rangle|\\
&  \leq \left(  \sum_{i=1}^{n}\Vert x_{i}\Vert^{q}\right)  ^{\frac{1}{q}%
}.\left(  \sum_{i=1}^{n}|\langle T(\frac{x_{i}}{\Vert x_{i}\Vert}),y_{i}%
^{\ast}\rangle|^{q^{\ast}}\right)  ^{\tfrac{1}{q^{\ast}}}\\
&  \leq \Vert T\Vert_{\Upsilon_{p}}.\Vert(x_{i})_{i=1}^{n}\Vert_{q}.\Vert
(y_{i}^{\ast})_{i=1}^{n}\Vert_{p^{\ast},|\omega|}.
\end{align*}
This implies that $T$ is positive strongly $(p,q)$-summing and $\Vert
T\Vert_{D_{p,q}^{+}}\leq \Vert T\Vert_{\Upsilon_{p,q}}.$\newline Conversely,
assume that $T$ is positive strongly  $(p,q)$-summing. Let $(z_{i})_{i=1}^{n}$
be a finite sequence in $B_{E}$ and $(y_{i}^{\ast})_{i=1}^{n}$ in $(Y^{\ast
})^{+}$, we have
\begin{align*}
\left(  \sum_{i=1}^{n}|\langle T(z_{i}),y_{i}^{\ast}\rangle|^{q^{\ast}%
}\right)  ^{\tfrac{1}{q^{\ast}}} &  =\sup_{(\lambda_{i})_{i=1}^{n}\in
B_{\ell_{q}}}|\sum_{i=1}^{n}\lambda_{i}\langle T(z_{i}),y_{i}^{\ast}\rangle|\\
&  =\sup_{(\lambda_{i})_{i=1}^{n}\in B_{\ell_{q}}}|\sum_{i=1}^{n}\langle
T(\lambda_{i}z_{i}),y_{i}^{\ast}\rangle|\\
&  \leq \Vert T\Vert_{D_{p,q}^{+}}.\sup_{(\lambda_{i})_{i=1}^{n}\in B_{\ell
_{q}}}\Vert(\lambda_{i}z_{i})_{i=1}^{n}\Vert_{q}.\Vert(y_{i}^{\ast})_{i=1}%
^{n}\Vert_{p^{\ast},|\omega|}\\
&  \leq \Vert T\Vert_{D_{p,q}^{+}}.\sup_{(\lambda_{i})_{i=1}^{n}\in B_{\ell
_{q}}}\Vert(\lambda_{i})_{i=1}^{n}\Vert_{q}.\Vert(y_{i}^{\ast})_{i=1}^{n}%
\Vert_{p^{\ast},|\omega|}\\
&  =\Vert T\Vert_{D_{p,q}^{+}}.\Vert(y_{i}^{\ast})_{i=1}^{n}\Vert_{p^{\ast
},|\omega|}.
\end{align*}
This means that $T$ is positive  $(p,q)$-majorizing and $\Vert T\Vert
_{\Upsilon_{p,q}}\leq \Vert T\Vert_{D_{p,q}^{+}}.$
\end{proof}

\begin{corollary}
$T\in \mathcal{L}(E,Y)$ is positive $p$-majorizing if and only if $T$ is
positive strongly $p$-summing.
\end{corollary}

\section{Positive Cohen $(p,q)$-nuclear operators}

Cohen \cite{Cohen} introduced the concept of $p$-nuclear operators, which was
extended to the Cohen $(p,q)$-nuclear operators by Apiola \cite{Apiola}. Let
$1<p,q\leq \infty.$ An operator $T\in \mathcal{L}\left(  E,F\right)  $ is Cohen
$(p,q)$-nuclear if $\left(  T\left(  x_{n}\right)  \right)  _{n}\in \ell
_{p}\left \langle F\right \rangle $ whenever $\left(  x_{n}\right)  _{n}\in
\ell_{q,\omega}\left(  E\right)  $. We denote the space of Cohen
$(p,q)$-nuclear operators by $\mathcal{CN}_{\left(  p,q\right)  }\left(
E,F\right)  $. According to \cite{Apiola, Cohen}, the following conditions are
equivalent for a linear operator $T\in \mathcal{L}\left(  E,F\right)  $.%

\begin{equation}
T\in \mathcal{CN}_{\left(  p,q\right)  }\left(  E,F\right)  \Longleftrightarrow
\widehat{T}\in \mathcal{L}\left(  \ell_{q,\omega}\left(  E\right)  ,\ell
_{p}\left \langle F\right \rangle \right)  , \label{Theo App}%
\end{equation}
where $\widehat{T}\left(  \left(  x_{n}\right)  _{n}\right)  =\left(  T\left(
x_{n}\right)  \right)  _{n},$ for every $\left(  x_{n}\right)  _{n}\in
\ell_{q,\omega}\left(  E\right)  .$

In this section, we introduce the positive Cohen $(p,q)$-nuclear operators.
For $p=q$, these operators are closely linked to positive strongly $p$-summing
and positive $p$-summing operators, as stated in Kwapien's Factorization
Theorem (see \cite[Proposition 2]{Kw}). Here, we distinguish three cases:

\begin{definition}
\label{Def-coh}Let $1\leq q\leq p<\infty$ and $X$, $Y$ be Banach lattices, $E$
and $F$ be Banach spaces.

\begin{enumerate}
\item[(a)] \cite{these schoe} \textit{An operator }$T$ from a Banach lattice
$X$ to a Banach space $F$ is left positive Cohen $(p,q)$-nuclear if there
exists a constant $C>0$ such that for all $(x_{i})_{i=1}^{n}\subset X$ (of
positive elements), we have \newline%
\begin{equation}
\Vert(T(x_{i}))_{i=1}^{n}\Vert_{\ell_{p}\langle F\rangle}\leq C\Vert
(x_{i})_{i=1}^{n}\Vert_{q,|\omega|}. \label{cas1}%
\end{equation}

\item[(b)] \textit{An operator }$T$ from a Banach space $E$ to a Banach
lattice $Y$ is right positive Cohen $(p,q)$- nuclear if there exists a
constant $C>0$ such that for all $(x_{i})_{i=1}^{n}\subset E$, we have
\newline%
\begin{equation}
\Vert(T(x_{i}))_{i=1}^{n}\Vert_{\ell_{p}^{\pi}(Y)}\leq C\Vert(x_{i})_{i=1}%
^{n}\Vert_{q,\omega}. \label{cas2}%
\end{equation}

\item[(c)] \textit{An operator }$T$ from a Banach lattice $X$ to a Banach
lattice $Y$ is positive Cohen $(p,q)$- nuclear if there exists a constant
$C>0$ such that for all $(x_{i})_{i=1}^{n}\subset X$, we have \newline%
$\Vert(T(x_{i}))_{i=1}^{n}\Vert_{\ell_{p}^{\pi}(Y)}\leq C\Vert(x_{i}%
)_{i=1}^{n}\Vert_{q,|\omega|}.$ (see \cite[Definition 3.1]{chen2} for $r=1).$

The class of all positive Cohen $(p,q)$-nuclear operators from $X$ to $Y$
(respectively $X$ to $F$ and $E$ to $Y$ ) is denoted by $\mathcal{CN}%
_{p,q}^{+}$ (respectively $\mathcal{CN}_{p,q}^{left,+}(X,F)$ and
$\mathcal{CN}_{p,q}^{right,+}$)\newline we put $\Vert T\Vert_{\mathcal{CN}%
_{p}^{+}}=\inf C.$
\end{enumerate}
\end{definition}

The proof of the following results follows similar lines as in Proposition
\ref{c-the} and Proposition 3.22 in \cite{these schoe} and is omitted.

\begin{proposition}
\label{pro-coh}Let $1\leq q\leq p<\infty$ and $X$, $Y$ be Banach lattices, $E$
and $F$ be Banach spaces.

\begin{enumerate}
\item[(1)] \cite[Proposition 3.22]{these schoe} $T\in \mathcal{CN}%
_{p,q}^{left,+}(X,F);$ if and only if $\widehat{T}:\ell_{q,|\omega
|}(X)\longrightarrow \ell_{p}\langle F\rangle$ is is a well-defined continuous
linear operator.

\item[(2)] $T\in \mathcal{CN}_{p,q}^{right,+}(E,Y),$ if and only if
$\widehat{T}:\ell_{q,\omega}(E)\longrightarrow \ell_{p}^{\pi}(Y)$ is is a
well-defined continuous linear operator.

\item[(3)] $T\in \mathcal{CN}_{p,q}^{+}(X,Y),$ if and only if $\widehat{T}%
:\ell_{q,|\omega|}(X)\longrightarrow \ell_{p}^{\pi}(Y)$ is is a well-defined
continuous linear operator.
\end{enumerate}
\end{proposition}

A result by Apiola states that the adjoint of a Cohen $(p,q)$-nuclear linear
operator is Cohen $(q^{\ast},p^{\ast})$-nuclear linear operator. When $p=q$,
this result appears in \cite{Cohen}. Utilizing Theorem \ref{lattice dual},
(\ref{str-dual}) and (\ref{b-dual}) and taking into account that the adjoint
of the operators $\widehat{T}:\ell_{q,|\omega|}(X)\longrightarrow \ell
_{p}\langle F\rangle,$ $\widehat{T}:\ell_{q,\omega}(E)\longrightarrow \ell
_{p}^{\pi}(Y)$ and $\widehat{T}:\ell_{q,|\omega|}(X)\longrightarrow \ell
_{p}^{\pi}(Y)$ can be identified with the operators%
\[
\widehat{T^{\ast}}:\ell_{p^{\ast},\omega}\left(  F^{\ast}\right)
\longrightarrow \ell_{q^{\ast}}^{\pi}(X^{\ast}),\widehat{T^{\ast}}%
:\ell_{q,|\omega|}\left(  Y^{\ast}\right)  \longrightarrow \ell_{q^{\ast}%
}\left \langle E^{\ast}\right \rangle \text{ and }\widehat{T^{\ast}}%
:\ell_{p^{\ast},|\omega|}\left(  Y^{\ast}\right)  \longrightarrow \ell
_{q^{\ast}}^{\pi}(X^{\ast}),
\]
defined as $\widehat{T^{\ast}}((x_{n}^{\ast})_{n})=(T^{\ast}(x_{n}^{\ast
}))_{n}$, we extend this to positive Cohen $(p,q)$-nuclear operators.

\begin{theorem}
Let $1\leq q\leq p<\infty$ and $X$, $Y$ be Banach lattices, $E$ and $F$ be
Banach spaces .

\begin{enumerate}
\item[(1)] The operator $T$ belongs to $\mathcal{CN}_{p,q}^{left,+}(X,F)$ if
and only if its adjoint $T^{\ast}$ belongs to $\mathcal{CN}_{q^{\ast},p^{\ast
}}^{right,+}(F^{\ast},X^{\ast})$. Furthermore,
\[
\Vert T\Vert_{\mathcal{CN}_{p,q}^{left,+}}=\Vert T^{\ast}\Vert_{\mathcal{CN}%
_{q^{\ast},p^{\ast}}^{right,+}}.
\]

\item[(2)] The operator $T$ belongs to $\mathcal{CN}_{p,q}^{right,+}(E,Y)$ if
and only if its adjoint $T^{\ast}$ belongs to $\mathcal{CN}_{q^{\ast},p^{\ast
}}^{left,+}(Y^{\ast},E^{\ast})$. Furthermore,
\[
\Vert T\Vert_{\mathcal{CN}_{p,q}^{right,+}}=\Vert T^{\ast}\Vert_{\mathcal{CN}%
_{q^{\ast},p^{\ast}}^{left,+}}.
\]

\item[(3)] The operator $T$ belongs to $\mathcal{CN}_{p,q}^{+}(X,Y)$ if and
only if its adjoint $T^{\ast}$ belongs to $\mathcal{CN}_{q^{\ast},p^{\ast}%
}^{+}(Y^{\ast},X^{\ast})$. Furthermore,
\[
\Vert T\Vert_{\mathcal{CN}_{p,q}^{+}}=\Vert T^{\ast}\Vert_{\mathcal{CN}%
_{q^{\ast},p^{\ast}}^{+}}.
\]

\end{enumerate}
\end{theorem}

\begin{remark}
In a recent paper \cite[Definition 3.1]{chen2}, the authors introduced the
concept of positive $(p,q)$-dominance, where $1/p+1/q=1/r$, defined between
Banach lattices. Within this framework, both the Pietsch's Domination Theorem
and Kwapien's Factorization Theorem are established. This concept precisely
aligns with the positive Cohen $p$-nuclear concept presented here when $r=1$.
Thus, by referring the reader to the papers \cite[Theorem 3.3 and Theorem
3.7]{chen2}, we can also derive the well-known theorems, namely Pietsch's
Domination Theorem and Kwapien's Factorization Theorem, for the other two
concepts proposed here (for left and right positive Cohen $p$-nuclear). Notice
that Kwapien's Factorization Theorem ensures that positive Cohen $p$-nuclear
are closely related to positive strongly $p$-summing and positive $p$-summing operators.
\end{remark}

\section{Tensor characterizations}

Now we are interested to characterize the aforementioned classes using
abstract summability properties linked to the continuity of tensor product
operators defined within vector-valued sequence spaces.

\textit{The Wittstock injective tensor Product and Fremlin projective tensor
Product. }For Banach lattices $X$ and $Y$, let $X\otimes Y$ denote the
algebraic tensor product of $X$ and $Y$ . For each $u=\sum_{i=1}^{n}%
x_{i}\otimes y_{i}\in X\otimes Y$ , define $T_{u}:X^{\ast}\rightarrow Y$ by
$T_{u}(x^{\ast})=\sum_{i=1}^{n}x^{\ast}(x_{i})y_{i}$ for each $x^{\ast}\in
X^{\ast}$. The injective cone on $X\otimes Y$ is defined to be%

\[
C_{i}=\left \{  u=\sum_{i=1}^{n}x_{i}\otimes y_{i}\in X\otimes Y:T_{u}(x^{\ast
})\in Y_{+},\forall x^{\ast}\in X_{+}^{\ast}\right \}  .
\]
Wittstock \cite{Wist1, Wist2} introduced the positive injective tensor norm on
$X\otimes Y$ as follows:%

\[
\left \Vert u\right \Vert _{i}=\inf \left \{  \sup \left \{  \left \Vert
T_{v}(x^{\ast})\right \Vert :x^{\ast}\in B_{X_{+}^{\ast}}\right \}  :v\in
C_{i},u\pm u\in C_{i}\right \}
\]
Let $X\widetilde{\otimes}_{i}Y$ denote the completion of $X\otimes Y$ with
respect to $\left \Vert \cdot \right \Vert _{i}$. Then $X\widetilde{\otimes}%
_{i}Y$ with $C_{i}$ as its positive cone is a Banach lattice (also see
\cite[Sect. 3.8 ]{Mey91}), called the Wittstock injective tensor product of
$X$ and $Y$. The projective cone on $X\otimes Y$ is defined to be%

\[
C_{p}=\left \{
{\displaystyle \sum \limits_{i=1}^{n}}
x_{i}\otimes y_{i}:x_{i}\in X_{+},y_{i}\in Y_{+},n\in \mathbb{N}\right \}  .
\]
Fremlin \cite{Fre1, Fre2} introduced the positive projective tensor norm on
$X\otimes Y$ as follows:%

\[
\left \Vert u\right \Vert _{\left \vert \pi \right \vert }=\sup \left \{  \left \vert
{\displaystyle \sum \limits_{i=1}^{n}}
\phi(x_{i},y_{i})\right \vert :u=%
{\displaystyle \sum \limits_{i=1}^{n}}
x_{i}\otimes y_{i}\in X\otimes Y,\phi \in M\right \}
\]
where $M$ is the set of all positive bilinear functional $\phi$ on $X\times Y$
with $\left \Vert \phi \right \Vert \leq1$. Let $X\widehat{\otimes}_{F}Y$ denote
the completion of $X\otimes Y$ with respect to $\left \Vert \cdot \right \Vert
_{\left \vert \pi \right \vert }$. Then $X\widehat{\otimes}_{F}Y$ with $C_{p}$ as
its positive cone is a Banach lattice (also see \cite[Sect. 3.8 ]{Mey91}),
called the Fremlin projective tensor product of $X$ and $Y$. Let $p$ be real
numbers such that $1<p<\infty,$ then, due to \cite[ ]{QG, Lab 04} we have

$(P_{1})$ $\ell_{p,|\omega|}^{u}(X)$ is isometrically lattice isomorphic to
$\ell_{p}\widetilde{\otimes}_{i}X$.

$(P_{2})$ $\ell_{p}^{\pi}(X)$ is isometrically lattice isomorphic to $\ell
_{p}\widehat{\otimes}_{F}X$.

Let $\ell_{p}\widehat{\otimes}_{\epsilon}E$ and $\ell_{p}\widehat{\otimes
}_{\pi}E$ denote the Grothendieck injective and projective tensor product of
$\ell_{p}$ with a Banach space $E$, respectively (see Ryan \cite{Rya}). It is
well known that the space $\ell_{p,\omega}^{u}(E)$ is isometrically isomorphic
to $\ell_{p}\widehat{\otimes}_{\epsilon}E$ whereas $\ell_{p}(X)$ is
isometrically isomorphic to $\ell_{p}\otimes_{\Delta_{p}}X$.(see
\cite[12.9]{teso} and $\ell_{p}\left \langle E\right \rangle $ is isometrically
isomorphic to $\ell_{p}\widehat{\otimes}_{\pi}E$ (see \cite[Proposition 2.2.5
and Proposition 2.2.6]{Cohen}, \cite[Corrolary 3.9]{Fourier} and \cite{Bu D}).
Given a linear operator $T:X\rightarrow Y$, its associated tensor product
operator $I\otimes T:\ell_{p}\otimes X\rightarrow \ell_{p}\otimes Y$ is defined
by
\[
I\otimes T(\sum_{i=1}^{n}e_{i}\otimes x_{i}):=\sum_{i=1}^{n}e_{i}\otimes
T(x_{i}),
\]
and this map is clearly linear.

We apply now Theorem \ref{u-the} and $(P_{1})$ to the class of positive
$(p,q)$-summing operators to get new characterizations in terms of tensor
product transformations\textbf{.}

\begin{corollary}
Let $1<p<\infty$ and $T\in{\mathcal{L}}(X,F)$. The following are equivalent:

\begin{enumerate}
\item $T$ is positive $(p,q)$-summing operator.

\item The induced linear operator $I\otimes T:\ell_{q}\widetilde{\otimes}%
_{i}X\rightarrow \ell_{p}\widehat{\otimes}_{\pi}F$ is continuous.
\end{enumerate}

In this case $\Vert T\Vert_{\Lambda_{p,q}}=\Vert I\otimes T\Vert$.
\end{corollary}

According to $(P_{1})$ and Theorem \ref{c-the}, we obtain characterizations in
terms of tensor product transformations for the class of positive strongly
$(p,q)$-summing operators

\begin{corollary}
Let $1<p\leq \infty$ and $T\in \mathcal{L}(E,Y)$. The following are equivalent:

\begin{enumerate}
\item $T$ is positive strongly $p$-summing.

\item The induced linear operator $I\otimes T:\ell_{p}\widehat{\otimes
}_{\Delta_{p}}E\rightarrow \ell_{p}\widehat{\otimes}_{F}Y$ is continuous.
\end{enumerate}

In this case $\Vert T\Vert_{\mathcal{D}_{p,q}^{+}}=\Vert I\otimes T\Vert$.
\end{corollary}

It is known from \cite[Theorem 2.1.3]{Cohen} that $T\in \mathcal{L}\left(
E,F\right)  $ is Cohen $p$-nuclear if and only if the mapping $I\otimes
T:\ell_{p}\widehat{\otimes}_{\epsilon}E\rightarrow \ell_{p}\widehat{\otimes
}_{\pi}F$\ is continuous. Utilizing Proposition \ref{pro-coh}, $(P_{1})$ and
$(P_{2})$, we extend this result as follows.

\begin{corollary}
Let $1\leq q\leq p<\infty$ and $X$, $Y$ be Banach lattices, $E$ and $F$ be
Banach spaces.

$(a_{1})$ $T\in \mathcal{L}\left(  X,F\right)  $ is left positive Cohen
$(p,q)$-nuclear if and only if the mapping $I\otimes T:\ell_{q}\widetilde
{\otimes}_{i}X\rightarrow \ell_{p}\widehat{\otimes}_{\pi}F$\ is continuous.

$(a_{2})$ $T\in \mathcal{L}\left(  E,Y\right)  $ is positive right Cohen
$(p,q)$- nuclear if and only if the mapping $I\otimes T:\ell_{q}%
\widehat{\otimes}_{\epsilon}E\rightarrow \ell_{p}\widehat{\otimes}_{F}Y$\ is continuous.

$(a_{3})=T\in \mathcal{L}\left(  X,Y\right)  $ is positive Cohen $(p,q)$-
nuclear if and only if the mapping $I\otimes T:\ell_{q}\widetilde{\otimes}%
_{i}X\rightarrow \ell_{p}\widehat{\otimes}_{F}Y$\ is continuous.
\end{corollary}

\end{document}